\documentclass[USenglish]{article}	
\usepackage{hyperref,amsmath,amsthm,amsfonts,amssymb,amscd,enumerate,tikz}
\usetikzlibrary{positioning}
\usepackage{cite}
\usepackage{forloop}

\theoremstyle{definition}
\newtheorem{theorem}{Theorem}[section]

\newtheorem{lemma}[theorem]{Lemma}

\newtheorem{definition}[theorem]{Definition}

\newtheorem{cor}[theorem]{Corollary}
\newtheorem{prop}[theorem]{Proposition}

\newtheorem{example}[theorem]{Example}
\newtheorem{remark}[theorem]{Remark}
\newtheorem*{altdefofbuilding*}{Alternate definition of building}
\newtheorem*{theorem*}{Theorem}
\newtheorem*{lemma*}{Lemma}
\newtheorem*{cor*}{Corollary}

\newcommand{\graph}{\mathcal{G}}
\newcommand{\poset}{\mathcal{S}}
\newcommand{\bldg}{\Delta}
\newcommand{\Ztwo}{\mathbb{Z}_2}
\newcommand{\length}{\ell}

\newcommand{\End}{\textrm{End}}

\newcommand{\real}{\mathcal{U}}
\newcommand{\panel}{\mathcal{R}}

\newcommand{\Stab}{\textrm{Stab}}

\newcommand{\res}{\mathcal{R}}
\newcommand{\spaut}{\textrm{Aut}_{\textrm{sp}}}
\newcommand{\comments}[1]{}

\renewcommand{\t}[1]{\widetilde{#1}}

\newcommand{\tbldg}{\t{\bldg}}

\newcommand{\tW}{\t{W}}

\newcommand{\tE}{\t{E}}
\newcommand{\tG}{\t{G}}

\newcommand{\aut}{\textrm{Aut}}
\newcommand{\ball}{D}

\newcounter{vertexindex}
\newcommand{\chapter}{}
\title{Automorphism Groups of Graph Products of Buildings}
\author{Aliska Gibbins\\ agibbins@fit.edu\\Florida Institute of Technology}
\begin{document}
\maketitle

\newcommand{\hexagons}{\label{hexagons}
$\begin{array}{ll}
\begin{tikzpicture}[scale=2]
\fill (0,0)[black!20!white] -- (.5,1.732/2-.01) --(-.5, 1.732/2-.01) --cycle;
\draw (0,0) +(0:1) -- +(60:1) -- +(120:1) -- +(180:1) -- +(240:1)-- + (300:1)  -- cycle;
\draw (0,1.732/2) -- (0,-1.732/2);
\draw (1,0) -- (-1,0);
\draw[rotate=60] (1,0) -- (-1,0);
\draw[rotate=120] (1,0) -- (-1,0);
\draw[rotate=60] (0,1.732/2) -- (0,-1.732/2);
\draw[rotate=120] (0,1.732/2) -- (0,-1.732/2);
\node at (0,1.723/2+.2) {$\emptyset$};
\node at (-.5,1.723/2+.2) {$\{s\}$};
\node at (.5,1.723/2+.2) {$\{t\}$};
\node at (.2,-.1) [fill=white] {$\{s,t\}$};
\draw[very thick, dashed, red] (.5,1.732/2) -- (0,0) ;
\draw[very thick, blue] (0,0)-- (-.5,1.732/2);
\end{tikzpicture}\hspace{.25in}
&
\begin{tikzpicture}[scale=2]
\draw (0,0) +(0:1) -- +(60:1) -- +(120:1) -- +(180:1) -- +(240:1)-- + (300:1)  -- cycle;
\node at (0,1.723/2+.2) {$\emptyset$};
\node at (-.5,1.723/2+.2) {$\{s\}$};
\filldraw[blue,xshift=-2pt,yshift=-2pt](-.5,1.723/2) rectangle ++(4pt,4pt);
\filldraw[red] (.5,1.723/2) circle (.075) ;
\node at (.5,1.723/2+.2) {$\{t\}$};
\end{tikzpicture}
\end{array}$

}

\newcommand{\fanoplane}{
$\begin{array}{ll}
\begin{tikzpicture}
\draw (0,0) +(0:2) -- +(360/14:2) -- +(360/7:2) -- +(3*360/14:2) -- +(4*360/14:2)-- + (5*360/14:2)-- + (6*360/14:2)-- + (7*360/14:2)-- + (8*360/14:2)-- + (9*360/14:2)-- + (10*360/14:2)-- + (11*360/14:2)-- + (12*360/14:2)-- + (13*360/14:2)-- cycle;
\forloop{vertexindex}{1}{\value{vertexindex} < 8}{
\draw [rotate=2*\value{vertexindex}*360/14] (0,0)+(0:2)-- + (5*360/14:2);
}
\forloop{vertexindex}{1}{\value{vertexindex} < 8}{
\filldraw[rotate=2*\value{vertexindex}*360/14,xshift=-2pt,yshift=-2pt,blue] (2,0)  rectangle ++(4pt,4pt) ;
\filldraw[rotate=(2*\value{vertexindex}+1)*360/14,red] (2,0) circle (.075) ;
}
\end{tikzpicture}
&
\begin{tikzpicture}[scale=2]
\draw (-.5,0)--(.5,0);
\draw (.5,0)--(.75,.5);
\draw (.75,.5)--(.95,.3);
\draw (.75,.5)--(.95,.7);
\draw (.75,-.5)--(.95,-.3);
\draw (.75,-.5)--(.95,-.7);
\draw (-.75,-.5)--(-.95,-.3);
\draw (-.75,-.5)--(-.95,-.7);
\draw (-.75,.5)--(-.95,.3);
\draw (-.75,.5)--(-.95,.7);
\draw (.5,0)--(.75,-.5);
\draw (-.5,0)--(-.75,.5);
\draw (-.5,0)--(-.75,-.5);
\filldraw[red] (-.5,0) circle (.07) ;
\filldraw[blue,xshift=-2pt,yshift=-2pt] (.5,0)  rectangle ++(4pt,4pt) ;
\filldraw[red] (.75,.5,0) circle (.05) ;
\filldraw[red] (.75,-.5) circle (.05) ;
\filldraw[blue,xshift=-1.5pt,yshift=-1.5pt] (-.75,.5) rectangle ++(3pt,3pt)  ;
\filldraw[blue,xshift=-1.5pt,yshift=-1.5pt] (-.75,-.5) rectangle ++(3pt,3pt)  ;
\filldraw[blue,xshift=-1pt,yshift=-1.5pt] (.95,.7) rectangle ++(2pt,2pt)  ;
\filldraw[blue,xshift=-1pt,yshift=-1.5pt] (.95,-.7)  rectangle ++(2pt,2pt)  ;
\filldraw[red] (-.95,-.7) circle (.03) ;
\filldraw[red] (-.95,.7) circle (.03) ;
\filldraw[red] (-.95,.3) circle (.03) ;
\filldraw[red] (-.95,-.3) circle (.03) ;
\filldraw[blue,xshift=-1pt,yshift=-1.5pt](.95,-.3)  rectangle ++(2pt,2pt)  ;
\filldraw[blue,xshift=-1pt,yshift=-1.5pt](.95,.3)  rectangle ++(2pt,2pt)  ;

\end{tikzpicture}
\end{array}$
}

\newcommand{\fanoplaneapartment}{
\begin{tikzpicture}
\draw (0,0) +(0:2) -- +(360/14:2) -- +(360/7:2) -- +(3*360/14:2) -- +(4*360/14:2)-- + (5*360/14:2)-- + (6*360/14:2)-- + (7*360/14:2)-- + (8*360/14:2)-- + (9*360/14:2)-- + (10*360/14:2)-- + (11*360/14:2)-- + (12*360/14:2)-- + (13*360/14:2)-- cycle;
\draw [line width=.15cm, gray] (0,0) + (12*360/14:2)-- + (13*360/14:2) -- +(0:2) -- +(360/14:2);
\draw [line width=.05cm] (0,0) + (12*360/14:2)-- + (13*360/14:2) -- +(0:2) -- +(360/14:2) -- +(360/7:2) -- +(3*360/14:2) --cycle;
\forloop{vertexindex}{1}{\value{vertexindex} < 8}{
\filldraw[rotate=2*\value{vertexindex}*360/14,blue] (2,0) circle (.075) ;
\filldraw[rotate=(2*\value{vertexindex}+1)*360/14,red] (2,0) circle (.075) ;
\draw [rotate=2*\value{vertexindex}*360/14] (0,0)+(0:2)-- + (5*360/14:2);
}
\end{tikzpicture}
}
\newcommand{\treerootapartment}{
\begin{tikzpicture}
\draw[line width=.05cm] (0,0)--(240pt,0);
\draw[xshift=12,line width=.09cm] (0,0)--(240pt,0);

\node at (-.5,0){\dots};
\node at (265pt,0){\dots};

\draw[xshift=12,line width=.18cm,gray] (0,0)--(60pt,0);
\draw[xshift=12,line width=.18cm,gray] (120pt,0)--(180pt,0);
\draw[xshift=12,line width=.09cm] (0,0)--(240pt,0);
\draw[xshift=12, rotate=60] (0,0)--(1.5,0);
\node at (0,0)[xshift=12, rotate=60,yshift=.5cm,xshift=1.25cm,rotate=60,xshift=10]{\dots};
\node at (0,0)[xshift=249, rotate=60,yshift=.5cm,xshift=1.25cm,xshift=10]{\dots};
\draw[xshift=12, rotate=60] (1,0)--(1.25,.5);
\draw[xshift=12,red,fill=red] (0,0) circle (.13);
\draw[xshift=12](0,-15pt)--(0,15pt);
\node[xshift=12] at (-7pt,-15pt){$m$};
\node[xshift=12] at (13pt,6pt){$C$};
\draw[xshift=12,rotate=60,blue,fill=blue] (1,0) circle (.08);
\draw[xshift=32, rotate=60] (0,0)--(1.5,0);
\draw[xshift=32, rotate=60] (1,0)--(1.25,.5);
\draw[xshift=32,blue,fill=blue] (0,0) circle (.13);
\draw[xshift=32,rotate=60,red,fill=red] (1,0) circle (.08);
\draw[xshift=12+40, rotate=60] (0,0)--(1.5,0);
\draw[xshift=12+40, rotate=60] (1,0)--(1.25,.5);
\draw[xshift=12+40,red,fill=red] (0,0) circle (.13);
\draw[xshift=12+40,rotate=60,blue,fill=blue] (1,0) circle (.08);
\draw[xshift=32+40, rotate=60] (0,0)--(1.5,0);
\draw[xshift=32+40, rotate=60] (1,0)--(1.25,.5);
\draw[xshift=32+40,blue,fill=blue] (0,0) circle (.13);
\draw[xshift=32+40,rotate=60,red,fill=red] (1,0) circle (.08);
\draw[xshift=12+80, rotate=60] (0,0)--(1.5,0);
\draw[xshift=12+80, rotate=60] (1,0)--(1.25,.5);
\draw[xshift=12+80,red,fill=red] (0,0) circle (.13);
\draw[xshift=12+80,rotate=60,blue,fill=blue] (1,0) circle (.08);
\draw[xshift=32+80, rotate=60] (0,0)--(1.5,0);
\draw[xshift=32+80, rotate=60] (1,0)--(1.25,.5);
\draw[xshift=32+80,blue,fill=blue] (0,0) circle (.13);
\draw[xshift=32+80,rotate=60,red,fill=red] (1,0) circle (.08);
\draw[xshift=12+120, rotate=60] (0,0)--(1.5,0);
\draw[xshift=12+120, rotate=60] (1,0)--(1.25,.5);
\draw[xshift=12+120,red,fill=red] (0,0) circle (.13);
\draw[xshift=12+120,rotate=60,blue,fill=blue] (1,0) circle (.08);
\draw[xshift=32+120, rotate=60] (0,0)--(1.5,0);
\draw[xshift=32+120, rotate=60] (1,0)--(1.25,.5);
\draw[xshift=32+120,blue,fill=blue] (0,0) circle (.13);
\draw[xshift=32+120,rotate=60,red,fill=red] (1,0) circle (.08);
\draw[xshift=12+160, rotate=60] (0,0)--(1.5,0);
\draw[xshift=12+160, rotate=60] (1,0)--(1.25,.5);
\draw[xshift=12+160,red,fill=red] (0,0) circle (.13);
\draw[xshift=12+160,rotate=60,blue,fill=blue] (1,0) circle (.08);
\draw[xshift=32+160, rotate=60] (0,0)--(1.5,0);
\draw[xshift=32+160, rotate=60] (1,0)--(1.25,.5);
\draw[xshift=32+160,blue,fill=blue] (0,0) circle (.13);
\draw[xshift=32+160,rotate=60,red,fill=red] (1,0) circle (.08);

\draw[xshift=12+200, rotate=60] (0,0)--(1.5,0);
\draw[xshift=12+200, rotate=60] (1,0)--(1.25,.5);
\draw[xshift=12+200,red,fill=red] (0,0) circle (.13);
\draw[xshift=12+200,rotate=60,blue,fill=blue] (1,0) circle (.08);
\draw[xshift=32+200, rotate=60] (0,0)--(1.5,0);
\draw[xshift=32+200, rotate=60] (1,0)--(1.25,.5);
\draw[xshift=32+200,blue,fill=blue] (0,0) circle (.13);
\draw[xshift=32+200,rotate=60,red,fill=red] (1,0) circle (.08);

\end{tikzpicture}
}

\newcommand{\hextoRAB}{
$\begin{array}{lll}
\begin{tikzpicture}
\draw [rotate=60,dashed](-.5,.866)--(.5,.866);
\draw [rotate=120,dashed](-.5,.866)--(.5,.866);
\draw [rotate=180,dashed](-.5,.866)--(.5,.866);
\draw [rotate=180,dashed](-.5,.866)--(.5,.866);
\draw [rotate=-60,dashed](-.5,.866)--(.5,.866);
\draw [rotate=-120,dashed](-.5,.866)--(.5,.866);
\draw [rotate=0,dashed](-.5,.866)--(.5,.866);
\draw[rotate=60](0,0)--(0,.866);
\draw[rotate=120](0,0)--(0,.866);
\draw[rotate=180](0,0)--(0,.866);
\draw[rotate=240](0,0)--(0,.866);
\draw[rotate=300](0,0)--(0,.866);
\draw[rotate=0](0,0)--(0,.866);
\draw[rotate=30,red,very thick, dashed](0,0)--(0,1);
\draw[rotate=-30,blue,very thick](0,0)--(0,1);
\draw[rotate=90,blue,ultra thick](0,0)--(0,1);
\draw[rotate=-90,red,ultra thick, dashed](0,0)--(0,1);
\draw[rotate=150,red,very thick,dashed](0,0)--(0,1);
\draw[rotate=-150,blue,very thick](0,0)--(0,1);
\end{tikzpicture}

&\hspace{.5in}&

\begin{tikzpicture}
\draw[rotate=60](0,0)--(0,1);
\draw[rotate=120](0,0)--(0,1);
\draw[rotate=180](0,0)--(0,1);
\draw[rotate=240](0,0)--(0,1);
\draw[rotate=300](0,0)--(0,1);
\draw[rotate=0](0,0)--(0,1);
\draw[rotate=0,red,fill=red](0,0) circle(.1);
\end{tikzpicture}

\end{array}$
}

\newcommand{\cubenolink}{

\begin{tikzpicture}

\draw (0,0)--(-1,1);
\draw (0,0)--(-0,1);
\draw (0,0)--(1,1);
\draw[ultra thick](-1,1)--(-1,2);
\draw[ultra thick](-1,1)--(0,2);
\draw(0,1)--(-1,2);
\draw(0,1)--(1,2);
\draw(1,1)--(0,2);
\draw(1,1)--(1,2);
\draw[fill=gray,opacity=.4] (0,0)--(-1,1)--(-1,2)--(0,1)--(0,0);
\draw[fill=gray,opacity=.5] (0,0)--(-1,1)--(0,2)--(1,1)--(0,0);
\draw (-.45,.45)--(0,.7)--(.45,.45);
\draw[dashed] (-.45,.45)--(.45,.45);
\draw[fill=gray,opacity=.4] (0,0)--(0,1)--(1,2)--(1,1)--(0,0);
\draw[fill=black](0,0) circle(.06);
\draw[fill=black](-1,1) circle(.06);
\draw[fill=black](0,1) circle(.06);
\draw[fill=black](1,1) circle(.06);
\draw[fill=black](-1,2) circle(.06);
\draw[fill=black](0,2) circle(.06);
\draw[fill=black](1,2) circle(.06);
\end{tikzpicture}
}
\begin{abstract}
\hspace{.2in} 
We begin by describing an underlying right angled building structure of any graph product of buildings.  We define the group of \emph{structure preserving automorphisms} of such an underlying right angled building and show that this group is the automorphism group of the graph product of buildings.  Theorem \ref{thm:ext} gives an explicit construction of the group of structure preserving automorphisms of a right angled building as an inverse limit of the structure preserving automorphism group of larger and larger combinatorial balls centered at a fixed chamber.

Finally in Section \ref{sec:GenGraphProd} we show that the notion of generalized graph product of a collection of groups $\{G_i\}_I$ from \cite{janswiat} with $G_i$ acting on $\bldg_i$ corresponds to the group of automorphisms of $\prod_\graph \bldg_i$ generated via lifts of the action of the product of the groups on $\prod_I \bldg_i$ up to a kernel of the $G_i$ actions.  We use these results to show that if each $\bldg_i$ is finite, with $G_i=\aut(\bldg_i)$ then the generalized graph product of the $G_i$ is residually finite.

\end{abstract}
20F65, 20F55, Graph Products, Buildings, Coxeter Groups, Reflection Groups
\section{Coxeter Groups}\label{CoxeterGroups}

Finite reflection groups were first studied in connection with Lie groups and Lie algebras. Coxeter classified both spherical reflection groups and cocompact Euclidean reflection groups in \cite{Cox34}.  The only other irreducible symmetric space is hyperbolic space, so any reflection group splits into factors of spherical groups, Euclidean groups and hyperbolic groups.  In \cite{TitsCoxeter} Tits introduced the following abstract notion of a reflection group, which is where we will begin.

A \emph{Coxeter matrix over a set $S=\{s_i\}_{i\in I}$}  is an $S \times S$ symmetric matrix $M = (m(s,t))$ with each diagonal entry equal to $1$ and each off-diagonal entry equal to either an integer $\geq 2$ or the symbol $\infty$.  The matrix $M$ determines a presentation of a group $W$ as follows: the set of generators is $S$ and the relations have the form $(st)^{m(s,t)}$ where $(s,t)$ ranges over all pairs in $S\times S$ such that $m(s,t)\neq \infty$. The pair $(W,S)$ is a \emph{Coxeter system} and $W$ is a \emph{Coxeter group}. The \emph{rank} of a Coxeter system is the cardinality of the generating set, $|S|$.  When the generating set is understood we will refer to this as the rank of $W$.

We say that a \emph{word} for $w$ in $(W,S)$ is a finite sequence $(s_1, s_2, \dots, s_n)$ with $s_i$ in $S$ such that $s_1 \cdot s_2 \cdots s_n=w$.  A word is said to be \emph{reduced} if it is of minimal length.  We denote $\ell(w)$ to be the length of any reduced word for $w$.  While there need not be a unique reduced word for $w$, the length of all reduced words for $w$ is the same.  This follows from \emph{Tits' solution to the word problem} \cite{titssol} which states that the only two moves necessary to reduce a word are 
\begin{enumerate}
\item replacing a subsequence of length $2k$ of the form $(s, t,s,t, \dots,s,t)$ with a subsequence of length $2(n-k)$ of the form $(t, s, t,s, \dots, t, s)$ when $m(s,t)=n$ in the Coxeter matrix, and 
\item removing any subsequence $(s,s).$
\end{enumerate}

Note that the definition above is independent of geometry.  Tits developed Coxeter groups in order to develop a unifying theory to study Lie groups and Lie algebras. Thus Coxeter groups were introduced as a purely group theoretical tool to work with.  We next define a special class of Coxeter group that will appear throughout.

A \emph{right angled Coxeter group} is a Coxeter group defined by a Coxeter matrix with entries either $1$, $2$ or the symbol $\infty$. That is, the only relations aside from the generators being involutions are of the form $(st)^2$ for $s,t$ any two generators in $S$.  We can construct such a Coxeter group from a graph in the following manner.  Given a simplicial graph $\graph$ with vertex set $S$ we define the \emph{right angled Coxeter group associated to $\graph$} as the Coxeter group generated by $S$ with entries in the Coxeter matrix given by $m(s,s)=1$ as required, $m(s,t)=2$ if $\{s,t\}$ is an edge in $\graph$ and $m(s,t)=\infty$ otherwise. Any Coxeter group which arises in this manner is a right angled Coxeter group and any right angled Coxeter group has a graph it is associated to, namely the graph with edge set $S$ and edges $\{s,t\}$ when $m(s,t)=2$. 

We say that $W$ is \emph{spherical} if it is finite.  We denote the subgroup of $W$ generated by $T \subset S$ as $W_T$.  We say that $T$ is \emph{spherical} if $W_T$ is spherical. Spherical subsets will become important later when we begin constructing appropriate geometries from these groups.  We begin by first defining abstract buildings.

\section{Buildings}\label{sec:bldg}
We will use the more traditional definition of building.

\begin{definition}\label{def:bldg}
Suppose $(W,S)$ is a Coxeter system.  A \emph{building of type $(W,S)$} is a pair $(\bldg,\delta)$ consisting of a nonempty set $\bldg$, whose elements are called \emph{chambers} and a function $\delta: \bldg \times \bldg \to W$ so that the following conditions hold for all chambers $C, D \in \bldg$.
\renewcommand{\theenumi}{(WD\arabic{enumi})}
\begin{enumerate}
\item\label{WD1} $\delta(C,D)=1$ if and only if $C=D$.
\item \label{WD2}If $\delta(C,D)=w$ and $C' \in \bldg$ satisfies $\delta(C',C)=s\in S$, then $\delta(C',D) = sw$ or $w$.  If in addition $\length(sw)= \length(w)+1$, then $\delta (C',D)=sw.$
\item \label{WD3} If $\delta(C,D)=w$, then for any $s \in S$ there is a chamber $C' \in \bldg$ such that $\delta(C',C)=s$ and $\delta(C',D)=sw$.
\end{enumerate}
\renewcommand{\theenumi}{\arabic{enumi}}
\end{definition}

A treatment of equivalent definitions can be found in \cite{abbrown}.  

We say two chambers $C$ and $D$ are \emph{$s$-adjacent} if $\delta(C,D)=s$ and we refer to $s$ as the \emph{type} of the adjacency.  We write $C \sim_s D$ if $C$ is \emph{$s$-equivalent} to $D$, that is if $C$ is $s$-adjacent to $D$ or $C=D$.   For any chamber $C$ and subset $T \subset S$ the \emph{$T$-residue containing $C$ or $\res_T(C)$} is the set of all chambers in $\bldg$ that can be connected to $C$ by adjacencies whose types lie only in $T$.  The sequence of chambers in such a path connecting $C$ to $C'$, $\{C_0=C,C_1,\dots, C_n=C'\}$ with $C_i \equiv_{s_i} C{i+1}$ with $s_i\in T$  is called a \emph{gallery of type $T$.}  If the set is understood, then we write the residue as $\res(C)$.  If we are choosing a residue without specifying a chamber it contains, we will write $\panel_T$ and if $T =\{s\}$ for some $s$, then we can write $\res_s (C)$, instead of $\res_T(C)$.  We will often  refer to such residues as the \emph{$s$-panel containing $C$}.

\begin{example}\label{Rank1}
Start with a Coxeter system $(W,S)$ where $S=\{s\}$ and $W =\Ztwo$.  A building $\bldg$ of type $(W,S)$ is just a set $\bldg$ with distance function $\delta(C,C')=s$ if $C\neq C'$ and $\delta(C,C)=1$.  We call such a building a \emph{rank-$1$ building.}
\end{example}

In general the \emph{rank of a building} $\bldg$ of type $(W,S)$ is just the cardinality of $S$.  The rank of a residue $\panel_T$ is just the cardinality of $T$.

\begin{example}\label{thin}
Fix some Coxeter system $(W,S)$.  Then $W$ is a building with $W$-distance $\delta(w,w')= w^{-1} w'$.  Thus two chambers $w$ and $w'$ are $s$-adjacent if $ws=w'$.  In such a building every panel has cardinality two.  Any building in which every panel has cardinality exactly two can be identified with its Coxeter group in this way.  Since every panel is of minimal size we call this building the \emph{thin building of type $(W,S)$}.
\end{example}

In most cases of interest to us at least some panels will have size larger than two.  If every panel in a building has size at least three, we say the building is \emph{thick}.  We will often be concerned with thick buildings and this makes it important to consider panels and, more generally, residues.  

\begin{prop}\label{subbldg}
Any residue of type $T$ is a building of type $(W_T,T)$.
\end{prop}

This follows from definitions.  Further, for any $C\in \bldg$ and $T\subset S$, we have that $C$ is contained in exactly one $T$-residue.  If $\panel$ is a $T$-residue, then $\panel$ can also be viewed as a building of type $W_T$.  In Example~\ref{thin} the subset corresponding to the subgroup $W_T$ is a $T$ residue, and thus is a thin building of type $(W_T,T)$. In this case the $T$ residues correspond to the right cosets of the subgroup $W_T$ in $W$.

\begin{prop}\label{bldg1}
For any building $\bldg$, any residue $\panel \subset \bldg$ and any chamber $C_0 \in \bldg$ there exists a unique chamber $C \in \panel$ such that $\ell(\delta(C,C_0))$ is minimal.
\end{prop}
This proposition follows from Tits' solution to the word problem.  The chamber $C$ is called the \emph{projection of $C_0$ onto $\panel$}.  
\begin{definition}\label{def:RAB}
A \emph{right angled building} is a building associated to some Coxeter system $(W,S)$ where $W$ is a right angled Coxeter group. We say a right angled building, is \emph{regular} if for every $s$ in $S$ the cardinality of every $s$-panel is constant.
\end{definition}

\begin{example} \label{tree}
An infinite tree $\mathcal{T}$ with no leaves can be viewed as a right angled building of type  $\Ztwo \star \Ztwo$ with chamber set the edge set of $\mathcal{T}$ and panels the set of edges in the star of a vertex, which we can identify with the vertices.  The type of the panels alternates between vertices.  If $\mathcal{T}$ is bi-regular then $\mathcal{T}$ is a regular right angled building.  
\end{example}

\begin{remark}\label{RABcard}
It is known that a regular right angled building is completely determined by its type and the cardinality of each $s$-panel.
This was written down by Haglund and Paulin in \cite{RegularRABUnique}.  Thus in Example~\ref{tree} combinatorially the trivalent tree is the unique building of type $\Ztwo \star \Ztwo$ with all panels of size three.
\end{remark}

The example of a tree gives a suggestion of the geometry which arises from this construction.  We will now make rigorous this identification.  Consider an arbitrary space $X$.  A \emph{mirror structure} over an arbitrary set $S$ on $X$ is a family of subspaces $(X_s)$ indexed by $S$.  The $X_s$ are called \emph{mirrors}.  Given a mirror structure on $X$, a subspace $Y \subset X$ inherits a mirror structure by setting $Y_s := Y \cap X_s$.  For each nonempty subset $T \subset S$, define subspaces $X_T$ and $X^T$ by 
$$ X_T:= \bigcap_{s\in T} X_s \textrm{ and } X^T:= \bigcup_{s\in T} X_s$$

Put $X_\emptyset := X$ and $X^\emptyset := \emptyset$. Given a subset $c$ of $X$ or a point $x \in X$ put 
\begin{eqnarray*}S(c) &=& \{s\in S\ |\ c\subset X_s\}\\
S(x) &=& \{s\in S\ |\ x\in X_s\}.
\end{eqnarray*}
A space with a mirror structure is a \emph{mirror space}.  Given a building $\bldg$ of type $(W,S)$ and a mirrored space $X$ over $S$ define an equivalence relation $\sim$ on $\bldg \times X$ by $(C,x) \sim (D,y)$ if and only if $x=y$ and $\delta(C,D) \in S(x)$. We define the \emph{$X$-realization} of $\bldg$, denoted $\real(\bldg,X)$ as
$$\real(\bldg,X):= (\bldg \times X) /\sim. $$

In words, we identify $s$-mirrors of two chambers when those chambers are $s$-adjacent. We define the realization of a chamber $C\in \bldg$ inside a realization $\real(\bldg,X)$ to be the image of $(C,X)$ and the realization of a residue as the union of the realizations of the chambers contained in that residue. Thus the realization of a $T$ residue in $\real(\bldg,X)$ is the $X$ realization of the residue as a building of type $(W_T,T)$. 

We say that two panels $\panel_s$ and $\panel_t$ in $\bldg$ are \emph{adjacent} if they are contained in a spherical residue of type $\{s,t\}$ and the set of distances between chambers in $\panel_s$ and $\panel_t$ is a coset of $W_{\{s\}}$ in $W_{\{s,t\}}$.  We extend these adjacencies to equivalence classes and define a \emph{wall} in $\bldg$ as an equivalence class of panels. By the realization of a wall we mean the union of the mirrors corresponding to the panels in the wall.  Let $\mathcal{D}$ be a convex subset of $\bldg$.  We define the \emph{boundary of $\mathcal{D}$} to be the set of all panels in $\bldg$ which have proper intersection with $\mathcal{D}$. 

We now concern ourselves with what it means for a realization of a building to be `nice.'  There are several desirable properties we would like our realizations to have.  Most notably, we would like the realization to be a simply connected simplicial complex.  In order for the realization to be a simplicial complex we must begin with $X$, the realization of a chamber, being a simplicial complex, though in general $X$ will not be a simplex.  For the realization to be simply connected we must make sure all spherical residues are simply connected.  To this end, denote the poset of spherical subsets of $S$ partially ordered by inclusion by $\poset(W,S)$, or $\poset$ if the Coxeter system is understood.  For any $T\subset S$ define $\poset_{\geq T}$ to be the poset of spherical subsets of $S$ which contain $T$.

\begin{definition}\label{daviscomplex}
Let $\poset$ be as above and put $X = |\poset|$, the geometric realization of the poset $\poset$.  (Recall that the geometric realization of a poset has simplices the chains in $\poset$.)  Define a mirror structure on $X$ as follows: for each $s \in S$ put $X_s:= |\poset _{\geq \{s\}}|$ and for each $T\in \poset$, let $X_T := | \poset_{\geq T}|$.  We say the complex $X$ with this mirror structure is the \emph{Davis chamber} of $(W,S)$.  In general we denote the Davis chamber as $K.$ 
\end{definition}

It follows from the Davis chamber being a realization of a poset that $K$ is a flag complex.  Often we determine the $1$-skeleton and define $K$ to be the flag complex with that $1$-skeleton.

The \emph{nerve} of $(W,S)$, written $L(W,S)$, is the poset of nonempty elements in $\poset$.  It is an abstract simplicial complex.  Thus for a Coxeter system $(W,S)$ we have a simplicial complex $L(W,S)$ with vertex set $S$ and whose simplicies are spherical subsets of $S$.  The Davis chamber can also be defined as the cone on the barycentric subdivision of $L(W,S)$.  These two definitions of the Davis chamber are equivalent.  The empty set in Definition~\ref{daviscomplex} corresponds to the cone point above. 

It is proved in \cite{davisbldgcat} that for any building $\bldg$ of type $(W,S)$ with $K$ its Davis chamber  $\real(\bldg,K)$ is contractible.  We say this is the \emph{standard realization} of $\bldg$.

\begin{example}
Let $W$ be the dihedral group of order six, with generating set $S=\{s,t\}$.  Let $\bldg$ be the thin building of type $(W,S)$ as in Example \ref{thin}.  Then the left hexagon in Figure \ref{hexagons} is the standard realization of $\bldg$ while the right hexagon is the realization with just a cone on the generators.  The mirrors in a single chamber in each are emphasized.\\
\end{example}

\begin{figure}[htb]
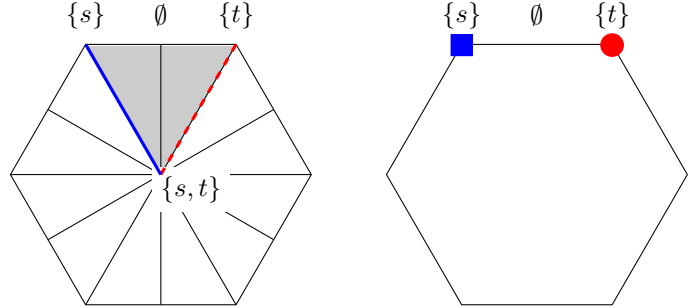

\centering
\hexagons
\caption{Two realizations of the thin building for the dihedral group of order six; the left with the Davis chamber, and the right with an edge.}
\label{hexagons}
\end{figure}

In \cite{exbldg} Davis introduced a construction of buildings from covering spaces.  We will outline this construction in \ref{sec:cover}.  Next we will discuss graph products of Coxeter groups and buildings as an application of this construction (also from \cite{exbldg}).  In Section \ref{sec:graphprodRAB} we will show that any graph product of buildings can be endowed with a new distance function making it a right angled building thus giving a realization of graph products of buildings as CAT(0) cube complexes.  We will then offer an immediate application of this fact by classifying when a graph product of buildings is Gromov hyperbolic as a generalization of Mousong's result in \cite{MoussongThesis}.

\section{Buildings Via Covering Spaces}\label{cover}\label{sec:cover}

We start with a Coxeter system $(W,S)$ with Coxeter matrix $(m(s,t))$ and a building $\bldg$ of type $(W,S)$.  Let $(m'(s,t))$ be a new $S\times S$ Coxeter matrix where $m'(s,t)=m(s,t)$ if $m'(s,t)\neq \infty$.  Denote the Coxeter system with the new Coxeter matrix as $(W',S)$.  

Let $K'$ be the Davis chamber for the Coxeter system $(W',S)$.  The non-standard realization of $\bldg$ with this Davis chamber, $\real(\bldg,K')$ is not simply connected unless $(m(s,t))=(m'(s,t))$.  In \cite{exbldg} Davis proved the following:

\begin{theorem}
The universal cover of $\real(\bldg,K')$, $\widetilde{\real(\bldg,K')}$, is the standard realization of a building of type $(W',S)$.
\end{theorem}

In light of this theorem we will usually denote $W'$ as $\tW$.  We let $\tbldg$ be the building with this realization.  Explicitly we define $\tbldg$ to be the set of copies of $K'$ in the realization $\widetilde{\real(\bldg,K')}$.  We define the $s$-adjacencies on $\tbldg$ to be those copies of $K'$ which share an $s$-mirror.  We can extend this to a $\t{W}$-metric.  The result is a building of type $(\t{W},S)$.

\begin{example}\label{ex:fanoplane}
Let $W$ be the dihedral group of order six, then the Coxeter matrix for $W$ is:
$$\left(\begin{array}{ll}
1&3\\
3&1
\end{array}
\right)$$
and let $\bldg$ be a building of type $(W,S)$ with each panel of size three.  Let $\tW$ be the Coxeter group with generating set $S$ and Coxeter matrix:

$$\left(\begin{array}{ll}
1&\infty\\
\infty&1
\end{array}
\right)$$

In Figure \ref{fanoplane} the realization on the left is $\real(\bldg,K')$ where $K'$ is the Davis chamber for the Coxeter system $(\tW,S)$. The resulting graph is the incidence graph of the Fano plane \cite{davisbook}.  The tree on the right is the universal cover of the realization on the left which from Example \ref{tree} we know is the realization of a building of type the infinite dihedral group.
\begin{figure}[htb]
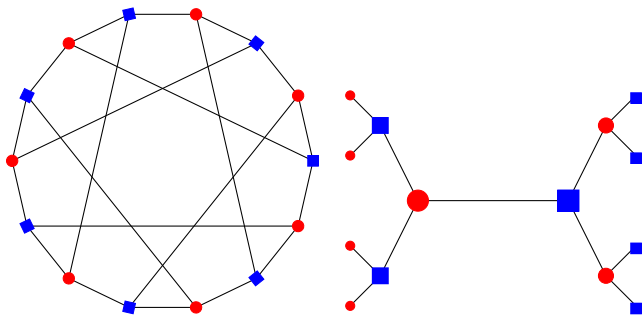

\centering
\fanoplane
\caption{The incidence graph of the Fano plane and its universal cover, the trivalent tree.}
\label{fanoplane}
\end{figure}
\end{example}

\begin{remark}\label{rmk:fundgroup}
Let $\pi$ be the fundamental group of $\real(\bldg,K')$.  Davis showed that there is a free action of $\pi$ on $\tbldg$ and that the quotient set is identified with $\bldg$. It is sometimes helpful to use the identification of $\tbldg$ with $\pi \times \bldg$.  Two chambers $(a,C)$ and $(b,D)$ are $s$-adjacent if $a=b$ and $C$ is $s$-adjacent to $D$.
\end{remark}

\section{Graph Products of Coxeter Groups}\label{sec:GraphProducts}

The classical ways to combine groups are with direct products and free products.  Graph products are an intermediate where we allow some factors to commute, but not others.  The first examples of graph products are right angled Coxeter groups, and right angled Artin groups.  The automorphism groups of each of these have been well studied.  

We we will procede by giving a formal definition of graph products of Coxeter groups.  This definition is equivalent to the quotient of the free product by relations where some sets of generators commute as in the more general case, which we will define in Section~\ref{sec:GenGraphProd}.  However, we find it useful to define graph products in terms of a Coxeter matrix to stay in the context of the construction given in Section \ref{cover}.  

For the rest of this paper let $\graph$ be a graph with vertex set $V(\graph)=I$.  Let $\{(W_i,S_i)\}_{i\in I}$ be a collection of Coxeter systems.  

Let $W$ be the product of the $W_i$, and $S$ the disjoint union of the $S_i$ so that $(W,S)$ is a Coxeter system with Coxeter matrix: 
$$
m(s,t) = \left \{
\begin{array}{ll}
 m_i(s,t) & \textrm{if $i=j$},\\
2 & \textrm{otherwise.}\\
\end{array}\right.$$

Then we apply the construction from Section \ref{cover} to this Coxeter system in the following way.

\begin{definition}\label{graphproductsCoxetergroups}
The \emph{graph product of the $W_i$ over $\graph$} is the Coxeter group with generating set $S=\bigsqcup _{i\in I} S_i$ and Coxeter matrix $\t{m}(s,t)$ with $s\in S_i$ and $t\in S_j$ where 
$$
\t{m}(s,t) = \left \{
\begin{array}{ll}
 m_i(s,t) & \textrm{if $i=j$},\\
2 & \textrm{if $i\neq j$ and $\{i,j\} \in E(\graph)$}\\
\infty & \textrm{otherwise}.
\end{array}\right.$$
We denote this Coxeter group as $\prod_\graph W_i$, or $\tW$ when all data is understood.
\end{definition}
In words, we include all of the relations from each of the $W_i$ and then include the additional relations that the elements of $S_i$ commute with those in $S_j$ exactly when $\{i,j\}$ is an edge in the graph. Our new Coxeter matrix has the matricies $m_i$ along the diagonal and off the diagonal there are blocks of either all $2$ or all the symbol $\infty$.  
 
\begin{example}\label{example:RACG}
The graph product of copies of $\Ztwo$, that is the graph product of Coxeter groups with just one generator, is a right angled Coxeter group.  Further, any right angled Coxeter group can be written as such a graph product.  Then the group $\prod_\graph \Ztwo$ is the right angled Coxeter group associated to $\graph$ as constructed at the end of Section~\ref{CoxeterGroups}.  In this paper we will use nonstandard notation for the sake of simplicity and take the generating set to be exactly the vertex set of $\graph$ rather than a set indexed by the vertices.
\end{example}

\section{Graph Products of Buildings}\label{sec:grprodbldg}

We now apply the construction for buildings from Section \ref{cover} to define the graph product of buildings. This was also due to Davis in \cite{exbldg}. We first give an explicit construction the Davis chamber for the graph product of Coxeter groups.  

Let $K_i$ denote the Davis chamber for each $(W_i,S_i)$ and let the nerve  $L(W_i,S_i)$ be denoted $L_i$ so that $K_i$ is the cone on the barycentric subdivision of $L_i$.  We define $L$ to be the nerve $L(\prod_\graph W_i, \bigsqcup S)$.  Then the $1$-skeleton of $L$ is formed from the disjoint union of the $L_i$ by attaching $1$-cells connecting vertices in $L_i$ with vertices in $L_j$ when $\{i,j\}$ is an edge in $\graph$.  Thus for every edge $\{i,j\}$ in $\graph$ we have that $L$ contains the complete bipartite graph between $L_i$ and $L_j$. Because $L$ is a flag complex it is completely determined by its one skeleton. Recall that the Davis chamber is the cone on the barycentric subdivision of $L$, the nerve of the Coxeter system.  The empty set in each $K_i$ is the cone point of that $K_i$, and in $K'$, the Davis chamber for $(\t{W},S)$, each of these cone points is identified.  

Note that $L$ is a subcomplex of $\prod_I L_i$, which means that $K'$ is contained in the Davis chamber of $(W,S).$  Thus we can also think of the construction as removing faces of $K$ when edges in $\graph$ are missing.

We now introduce new data: for each $i\in I$ let $\bldg_i$ be a building of type $(W_i,S_i)$.  First, $\bldg=\prod_I \bldg_i$ is a building of type $\prod_I W_i$.  Since every $S_i$ commutes with every $S_j$ when $i \neq j$ the $\prod_I W_i$ distance between two chambers is just the product of the distances between corresponding components. Call this product of buildings $\bldg$. We now apply the realization of Section \ref{cover} to $\bldg$.  Let $K'$ be the Davis chamber for $\prod_\graph W_i$ constructed above.  The realization $\real(\bldg,K')$ will not be simply connected (unless $\graph$ is the complete graph on $I$), since $K'$ is not the Davis chamber for $\prod_I W_i$.

Thus we can apply Theorem \ref{Prop:Cover} to graph products and gain the following:

\begin{prop}\label{Prop:Cover}
The universal cover of $\real(\bldg,K')$ is the standard realization of a building of type $\prod_\graph W_i$.
\end{prop}

Thus we define the graph product of buildings using the construction from Section \ref{sec:cover}.

\begin{definition}
The \emph{graph product of the $\bldg_i$ over $\graph$}, denoted $\prod_\graph \bldg_i$, is the set of copies of $K'$ in $\widetilde{\real(\bldg,K')}$, the universal cover of the $K'$ realization of $\bldg$, with two chambers $C$ and $C'$ being $s$-adjacent when their intersection is a copy of $K'_s$.  The $\tW$-distance is defined by the $s$-mirrors that a path between two centers intersects.
\end{definition}

When discussing graph products and all data is understood we will denote $\prod_I W_i$ as $W$ and $\prod_\graph W_i$ as $\tW$.  Also, $\bldg$ will denote $\prod_I \bldg_i$ and $\tbldg$ will denote $\prod_\graph \bldg_i$.

\begin{example}
If each of the $W_i$ are $\Ztwo$ then $\tW$ is a right angled Coxeter group and $\tbldg$ is a right angled building as in Definition \ref{def:RAB}.  Further each $i$-panel has cardinality $|\bldg_i|$ making $\tbldg$ a regular right angled building.  By Remark \ref{RABcard} any regular right angled building can be written as such a graph product.  In this case each of the $\bldg_i$ is a rank one building, which we can think of as just a set as in Example \ref{Rank1}.  
\end{example}

\begin{example}
If $\graph$ is the complete graph on $I$ then $\tbldg$ equals $\bldg$, the product of the $\bldg_i$.
\end{example}

\section{Graph Products of Buildings as Right Angled Buildings}\label{sec:graphprodRAB}

In general, any building can be viewed as a building of type $\Ztwo$ as in Example \ref{Rank1}. The set remains the same and we define the $\Ztwo$-distance between any two distinct chambers to be the non-identity element in the group.  

We will consider each $\bldg_i$ as a building of type $\Ztwo$.  To avoid confusion, we will let $E_i$ denote the set $\bldg_i$ with the $\Ztwo$ metric.  Because these are \emph{the same sets}, just with a different metric, there is a "natural" identification between them.  This identification induces a map between $\prod_I E_i$ and $\bldg$.  Let $E$ denote $\prod_I E_i$ and $K'_E$ denote the Davis chamber for $\prod_\graph \Ztwo$.  We know that $\tE=\prod_\graph E_i$ is a regular right angled building of type $\prod_\graph \Ztwo$.

\begin{theorem}\label{thm:mainresult}
There exists a unique isomorphism $\phi:\tE\to \tbldg$ (up to a choice of base chamber) so that the diagram commutes:\\
\begin{center}
\begin{tikzpicture}[node distance=2cm, auto]
  \node (P) {$\t{E}$};
  \node (B) [right of=P] {$\tbldg$};
  \node (A) [below of=P] {$E$};
  \node (C) [below of=B] {$\bldg$};
  \draw[-> , dashed] (P) to node {$\phi$}(B);
  \draw[->] (P) to node [swap] {$\displaystyle p_E$} (A);
  \draw[<->] (A) to node [swap] {$=$} (C);
  \draw[->] (B) to node {$p$} (C);
  \end{tikzpicture}\end{center}
\end{theorem}
\proof
Let $\poset_E=\poset(\prod_\graph \Ztwo, I)$ and $\poset =\poset(W,S)$. We construct a map $\bar{f}:\poset\to\poset_E$ in
the following way.  For any $T\subset S_i$ we let $\bar{f}(T)=\{i\}$.  If $T$ is a spherical subset which intersects more than one $S_i$ define $J=\{j|T\cap S_j\neq \emptyset\}$.  Because $T$ is spherical the elements of $J$ must span a complete subgraph in $\graph$.  Thus $J\in \poset_E$ and we let $\bar{f}(T)=J$.

Let $K'_E$ be the Davis chamber of $(\prod_\graph \Ztwo, I)$ and $K'$ the Davis chamber of $(W,S)$. The map $\bar{f}$ induces a map $f:K' \to K'_E$ by identifying the cone points and extending. Note that $K'_E \hookrightarrow K'$ where the $i$-mirror in $K'_E$ maps to the cone point of the spherical subset $S_i$.  These maps are a homotopy equivalence between $K'$ and $K'_E$.  We can extend this homotopy equivalence to $\real(K',\bldg)$ and $\real(K'_E,E)$ by applying the maps to the realization of every chamber.  The definition of the identifications guarantees that the extension of the homotopy is well defined and continuous. Thus $\real(K',\bldg)$ and $\real(K'_E,E)$ have the same fundamental group, $\pi$. Hence by Remark \ref{rmk:fundgroup} we have that $\tE$ is identified with $\pi \times E$ and $\tbldg$ is identified with $\pi \times \bldg$, but $\bldg = E$ so that $\tbldg$ and $\tE$ are isomorphic as sets. Thus we define our isomorphism as the identity on these identifications. If we fix an identification of $\tE$ with $\pi \times E$ and choose a different isomorphism $\phi'$ then we have a new identification of $\tbldg$ with $\pi \times \bldg$ which is equivalent to a different choice of base chamber.\endproof

\begin{example}
In Figure \ref{hextoRAB} we see the homotopy between the hexagon and the star on six vertices.  All of the mirrors on the left will map to the central vertex on the right. This central vertex corresponds to the $i$-mirror in the standard realization of $E_i$.
\begin{figure}[htb]
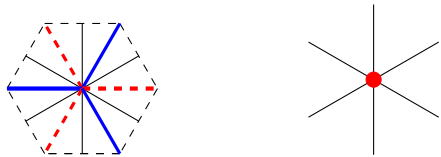

\centering
\hextoRAB
\caption{The standard realization of the Coxeter complex of the dihedral group of order six (left) and the standard realization of that set as a building of type $\Ztwo$ (right.)}
\label{hextoRAB}
\end{figure}
\end{example}
\begin{remark}\label{bldgEadj}
Note that two chambers in $\tE$ are $i$-adjacent exactly when they lie in the same $S_i$ residue in $\tbldg$. If we let $p:\tbldg \to \bldg$ be the map induced by the covering map from $\widetilde{\real(\bldg,K')}\to\real(\bldg,K')$ we see that $S_i$ residues in $\tbldg$ are isomorphic to $S_i$ residues in $\bldg$ under $p$.  This map extends to an isomorphism of $S_i \cup S_j$ residues exactly when $\{i,j\}$ is an edge in $\graph$ and so on.
\end{remark}

We will denote $S_i$ residues in $\tbldg$ by $\panel_i$.  This is a slight abuse of notation since $\panel_i$ also denotes an $i$-panel in $\tE$, however such residues are identified in the map from Theorem \ref{thm:mainresult}, so that it will not cause ambiguity in practice.  Further, for any $J\subset I$ we will say that any residue of type $T_J=\bigcup_{i\in J} S_i$ is a $J$-residue.

In \cite{CapraceSimple} Caprace proved the following lemma:

\begin{lemma}
Parallelism of residues in a building is an equivalence relation if and only if the building is right angled.
\end{lemma}

Applying this lemma to $\tE$ we have the following result:

\begin{lemma}
Paralellism of $J$-residues in $\tbldg$ is an equivalence relation.
\end{lemma}

\section{Hyperbolic Buildings}

Gromov had a notion of hyperbolicity which loosely meant that any triangle in the geodesic metric space $X$ is thin.  Such a space is said to be \emph{Gromov hyperbolic} \cite{bridson}.

Let $(W,S)$ be a Coxeter system.  Let $\bldg$ be the thin building of type $(W,S)$ as in Example~\ref{thin}, and $K$ be the cone on $S$ with mirrors $K_s=s$ for each $s$ in $S$.  The \emph{Cayley graph} of a Coxeter group $W$ with generating set $S$ is $\real(\bldg,K)$.  One can check easily that this is a rephrasing of the standard definition of the Cayley graph in building terminology.  We say that a group $G$ is \emph{hyperbolic} or \emph{word hyperbolic} if its Cayley graph is Gromov hyperbolic.

In \cite{MoussongThesis} Moussong proved the following result classifying word hyperbolic Coxeter groups:

\begin{theorem}\label{MoussongTheorem}
Let $(W,S)$ be a Coxeter system.  The following are equivalent:
\begin{enumerate}
\item $W$ is word hyperbolic.
\item $W$ has no subgroup isomorphic to $\mathbb{Z}\times \mathbb{Z}.$
\item $W$ does not contain a Euclidean sub-Coxeter system of rank greater than two and does not contain a pair of disjoint commuting Coxeter systems whose groups are both infinite.
\end{enumerate}
\end{theorem}

Later Moussong observed that a building has negative curvature exactly when its Coxeter group is word hyperbolic and this was written down by Davis in \cite{davisbldgcat}.  The notion of hyperbolicity has offered a rich class of examples in geometric group theory.  
 
Using these results and the isomorphism in Theorem \ref{thm:mainresult} we can completely categorize the hyperbolic buildings arising from graph products as a generalization of Theorem \ref{MoussongTheorem}.  For this we first need the graph $\graph$ to satisfy the \emph{no squares condition}, that is that every square in $\graph$ must have a diagonal inside.

\begin{theorem}
The standard realization of the graph product of buildings, $\prod_\graph \bldg_i$, is hyperbolic if and only if the following conditions hold:
\begin{enumerate}
\item \label{cond1} each $\bldg_i$ is spherical or infinite hyperbolic, 
\item \label{cond2} $\graph$ satisfies the no squares condition
\item \label{cond3} for all pairs of infinite hyperbolic buildings $\bldg_i$ and $\bldg_j$ we have that $\{i, j\}$ is not an edge in $\graph$ and
\item \label{cond4} if $\bldg_i$ is infinite hyperbolic then the star of $i$ spans a complete subgraph in $\graph$.
\end{enumerate}
\end{theorem}

\begin{proof}
We need only show that if each of these is satisfied then $\tW$ is hyperbolic.  Seeing that each of these is necessary is immediate.  By Moussong's Theorem we see that we need only guarantee that $\tW$ does not contain a Euclidean sub-Coxeter system of rank greater than two, guaranteed by conditions \ref{cond1} and \ref{cond3} and that $W$ does not contain a pair of disjoint commuting Coxeter systems whose groups are both infinite.  In order for a sub-Coxeter system to be infinite in $\tW$ it could contain an infinite hyperbolic $W_i$ or be the free product of two finite groups $W_i$ and $W_j$.  Conditions \ref{cond2}, \ref{cond3}, and \ref{cond4} verify that no such direct product of two of these exist.
\end{proof}
\section{Structure Preserving Automorphisms}\label{spaut}
The standard realization of a building $\bldg$ is locally finite if and only if each panel in $\bldg$ is finite. For this section we will assume that each panel in $\bldg$ is finite.  An automorphism $g$ of $\bldg$ is said to be \emph{type preserving} if for every $s$ in $S$ we have that $g$ sends $s$-panels to other $s$-panels.  If $g$ is an automorphism of a realization of a building then $g$ need not be type preserving, however we will only be considering actions on realizations of buildings that are type preserving, so in this paper an \emph{automorphism of a building $\bldg$} will always be a type preserving automorphism.  The group of all automorphisms of $\bldg$ is denoted $\aut(\bldg)$.  We will refer to an automorphism of a building and its induced map on the realization of the building interchangeably. 

Given the natural identification between graph products of buildings and right angled buildings, it makes sense to consider when an action on $\tE$ extends to an action on $\tbldg$ and vice versa.  Certainly the full automorphism group of $\tE$ will not act (as building automorphisms) on $\tbldg$ unless $\tbldg$ happens to be a graph product of rank-$1$ buildings.  We will show, however, that any building automorphism of $\tbldg$ is a building automorphism of $\tE$.  We will give a complete characterization of $\aut(\tbldg)$ as a subgroup of $\aut(\tE)$ and we will give a construction for the stabilizer of a chamber in each of these groups.

First, let $p:\tE\to E$ be the map induced by the covering map from $\widetilde{\real(E,K)}\to\real(E,K)$.  Further, let $p_i$ be $p$ composed with the projection onto $E_i$ in the product.  Then $p_i$ is a building isomorphism between any $S_i$ residue in $\tE$ and $E_i$.

Suppose that we are given as data a collection of groups $\{G_i\}_I$ with $G_i$ acting on $E_i$ for each $i\in I.$  We let $\prod_I G_i$ act on $E$ component wise.  Then we define our desired automorphisms of $\tE$ as follows:

\begin{definition}
We say that an automorphism $g$ of $\tE$ is \emph{structure preserving} if for each $i$-panel $\panel_i$ there exists a $g_i\in G_i$ such that $p_i\circ g(C) = g_i \circ p_i (C)$ for all $C$ in $\panel_i$.  We will denote the group of all structure preserving automorphisms of $\tE$ by $\spaut(\tE)$.
\end{definition}

We do not need to require the $E_i$ to be rank $1$ buildings, or $\tE$ to be right angled, however in practice we will only use this definition in the right angled case.  

\section{Stabilizer of a Chamber}
We next establish language to describe the full stabilizer of a chamber $C_0$ in $\spaut(\tE)$.  This construction is a generalization of the construction of the stabilizer of an edge in a tree as an inverse limit of iterated wreath products as in \cite{SerreTrees}.  For each $w\in \prod_\graph \Ztwo$ define $\End(w)$ to be the set of all $i \in I$ such that $\length(wi)<\length(w)$.   Define $\ball_n$ to be the set of all chambers $C$ such that $\length(\delta(C,C_0)) < n$.

If $g$ is an automorphism of $\ball_n$ such that for any $S_i$ residue that intersects $\ball_n$ nontrivially there exists $g_i\in G_i$ such that for each $C\in \bldg_i$ we have $p_i \circ g(C)= g_i \circ p_i(C)$ then $g$ is a \emph{structure preserving automorphism of $\ball_n$}.  The group of all structure preserving automorphisms of $\ball_n$ is denoted $\spaut(\ball_n)$.  We will determine $\spaut(\ball_n)$ recursively, first noting that $\spaut(\ball_0) = id$.  The restriction of $\spaut(\ball_{n+1})$ to $D_n$ maps $\spaut(\ball_{n+1})$ to $\spaut(\ball_n)$.  Let $K_{n+1}$ to be the kernel of the restriction.  We will show that the restriction is onto and describe the kernel.

We partition the set of chambers in $\ball = \ball_{n+1}\setminus \ball_n$ into two parts:
\begin{eqnarray*}
\mathcal{P}_1 &=&\{C\in \ball \vert\  |\End(\delta(C,C_0)|=1\}\\
\mathcal{P}_2 &=&\{C\in \ball \vert\  |\End(\delta(C,C_0)|\geq2\}.
\end{eqnarray*}

The sets $\mathcal{P}_1$ and $\mathcal{P}_2$ are disconnected.  If any two chambers $C_1$ and $C_2$ in $D$ are $i$-adjacent, then there is some unique $C$ in the $i$-panel that is in $\ball_n$ by Proposition \ref{bldg1}.  Thus $\delta(C_1,C_0)=\delta(C,C_0)\cdot i= \delta(C_2,C_0)$. Hence $\End(\delta(C_1,C_0))=\End(\delta(C_2,C_0))$ and $C_1$ and $C_2$ are in the same partition. It follows that there are no relations between chambers in $\mathcal{P}_1$ and $\mathcal{P}_2$. 

\begin{lemma} \label{lem:ext}
Let $\Phi_n$ denote the set $\{C\in D_n\vert C\in\panel_i(C') \text{ for some } C' \in \mathcal{P}_1\}$. The group $\spaut(D_n \cup \mathcal{P}_1)=\spaut(D_n) \rtimes \prod_{\Phi_n} \Stab_{G_i}(p_i(C))$ up to isomorphism.
\end{lemma}

\begin{proof}
Fix any $i$ panel $\panel$ intersecting $\mathcal{P}_1$. Let $C\in \panel$ be the projection of $C_0$ onto $\panel$ as in Proposition \ref{bldg1}. Then $C$ is the only chamber in $\panel$ contained in $\ball_{n}$. By the definition of $\spaut(\ball_n)$ for any element $g\in \spaut(\ball_n)$ we know there is at least one element $g_i\in G_i$ such that $p_i \circ g (C)=g_i \circ p_i (C)$ so that $g_i^{-1} \Stab_{G_i}(p(C)) g_i$ is an extension of $\spaut(D_n)$ to that $i$ panel.  This $g_i$ is not necessarily unique but a different choice induces an inner isomorphism.  All such panels are disjoint so the actions on panels will commute, thus completing the claim.
\end{proof}

\begin{lemma}
There is a unique extension of $\spaut(D_n)$ to $\spaut(D_n \cup \mathcal{P}_2)$.
\end{lemma}
\begin{proof}
Fix some $C_2 \in \mathcal{P}_2$ and let $\panel$ denote the $\End(\delta(C_2,C_0))$ residue containing $C_2$.  Let $C$ be the projection of $\panel$ onto $C_0$.  Now, $\panel$ is a right angled building of type $\prod_{\End(\delta(C_2,C_0))} \Ztwo.$  Thus it is isomorphic to its image under $p$ in $E$. The action on $\panel \cap \ball_n$ then corresponds to an element in $\prod G_i$ acting on $\prod E_i$, hence it can extend. Uniqueness follows in the same way since for any $g\in \spaut(\ball_n)$ we have that $g(C_2)$ must be $i$-adjacent to the projection of the $i$-panel onto $C_0$.  Any two such adjacencies completely determine the chamber in the product.
\end{proof}

Combining these we get the following result.
\begin{theorem}\label{thm:ext}
$\Stab_{\spaut(\tE)} (C_0)=\varprojlim \spaut(D_n)$
\end{theorem}

Thus we have that any structure preserving automorphism of a finite ball contained in $\tE$ can be extended to a structure preserving automorphism of the whole space.  

The previous theorem gives rise to the following results.
\begin{cor}
If $G_i$ acts freely on $E_i$ for all $i \in I$ then $\Stab_{\spaut(\tE)}(C_0)$ is trivial.
\end{cor}

\begin{cor}
If $i\in I$ is such that $<s_i>$ is not contained in a finite factor of $W$ and $G_i$ has non-trivial stabilizer at some element of $E_i$, then $\spaut(\tE)$ is uncountable.
\end{cor}

\begin{cor}
In particular, if for some $i\in I$ $<s_i>$ is not contained in a finite factor of $W$ and $E_i$ has at least three elements, then the full automorphism group of $\tE$, where each $G_i$ is the symmetric group on $|E_i|$, is uncountable.
\end{cor}

\section{Results for Graph Products}
We now consider structure preserving automorphisms in terms of graph products of buildings.  Fix $\graph$ a graph with vertex set $I$, as before.  Let $\{(W_i,S_i)\}_I$ be a collection of Coxeter systems with a collection of buildings $\{\bldg_i\}_I$ such that each $\bldg_i$ is a building of type $(W_i,S_i)$.  Let $\tbldg$ be the graph product of the $\bldg_i$ over $\graph$.  Further let $E_i$ be the rank one building with chamber set $\bldg_i$ and $\tE$ the graph product of the $E_i$ over $\graph$.

\begin{theorem}
For each $i\in I$ let $G_i=\aut(\bldg_i)$.  Then $\spaut(\t{E}) \cong \aut(\tbldg).$
\end{theorem}

\begin{proof}
From Theorem \ref{thm:mainresult} we have an isomorphism $\phi:\tbldg \to \tE$ which preserves projections.  The map $\phi$ induces an action of $\aut(\tbldg)$ on $\tE$.  The action of $\aut(\tbldg)$ preserves $i$-residues so that the induced map on $\tE$ is type preserving hence $\aut(\tbldg)$ injects into $\spaut(\tE)$.

We need to check that the action of $\spaut(\t{E})$ on $\tbldg$ is type preserving.  Fix some $g_{\tE}$ in $\spaut(\tE)$.  Let $g_{\tbldg}$ be the induced map on $\tbldg$ under $\phi$.  We verify that $g_{\tbldg}$ is a building automorphism.  Fix two chambers $C$ and $C'$ in $\tbldg$ with $C\sim_s C'$ for some $s\in S_i$.  We let $\phi(C)=C_E$ and $\phi(C')=C'_E$.  Then we have that $C_E \sim_i C'_E$.

Since $g_{\tE}$ is structure preserving there exists some $g_i \in G_i$ such that $p_i(g_{\t{E}}(C_E))=g_i(p_i(C_E))$ and $p_i(g_{\t{E}}(C'_E))=g_i(p_i(C'_E))$. Thus we have that $g_i(p_i(C))\sim_{s} g_i(p_i(C'))$ which implies that $p_i(g_{\tE}(C_E))\sim_s p_i(g_{\tE}(C'_E))$ in $\bldg_i$.  Hence $\phi(g_{\tE}(C_E))\sim_s \phi(g_{\tE}(C'_E))$ so that $g_{\tbldg}(C)\sim_s g_{\tbldg}(C')$ as desired.
\end{proof}

Thus we can consider the automorphism group of any graph product of buildings as a subgroup of the automorphism group of the underlying right angled building.  We can apply the previous results of right angled buildings to $\aut(\tbldg)$ to get the following:

\begin{cor}\label{cor:gruncount}
If there exists a thick $\bldg_i$ with $W_i$ not contained in a finite factor of $W$, then $\aut(\tbldg)$ is uncountable.
\end{cor}

\begin{cor} \label{cor:extend}
Any automorphism $g\in \aut(X)$ where $X$ is a combinatorial ball in $\tbldg$ centered at a chamber can be extended to an automorphism of $\tbldg$.
\end{cor}
 	
\section{Generalized Graph Product of Groups}\label{sec:GenGraphProd}

As before, let $\graph$ be a graph with vertex set $I$.  Let $\{E_i\}_{i\in I}$ be a collection of rank one buildings and $\{G_i\}_{i\in I}$ be a collection of groups with $G_i$ acting on $E_i$.  We are now going to require that these actions are \emph{transitive}.  Define $E =\prod_I E_i$ and $\tE = \prod_\graph E_i$ as before and let $W$ be the right angled Coxeter group associated to $\graph$ with generating set $I=V(\graph)$.

We will now define the \emph{generalized graph product} of groups acting on sets.  This definition is due to \cite{janswiat}.


Fix some $C=\prod_I C_i \in E$. Let $B_i$ be the stabilizer of $C_i$ in $G_i$ and $B = \prod_I B_i$.  Let $H_i$ be the product of the $B_k$ except with $G_i$ in the $i$-th position rather than $B_i$ and for all $\{i,j\}$ edges in $\graph$ define $H_{i,j}$ to be the product of the $B_k$ except with $G_i$ in the $i$-th position and $G_j$ in the $j$-th position.  We have inclusion maps $f_i:B \hookrightarrow H_i$, and $f_{i,j}: H_i \hookrightarrow H_{i,j}$ when $\{i,j\}$ is an edge in $\graph$.
\begin{definition}
We define the \emph{generalized graph product} of the $G_i$ over $\graph$ to be the direct limit of $B$, the $H_i$, and $H_{i,j}$ with inclusion maps $f_i$ and $f_{i,j}$.
\end{definition}

Call this generalized graph product $\overline{G}$.  Direct limits have the universal property that if any group $H$ has maps $h_i:H_i \to H$ and $h_{i,j}:H_{i,j}\to H$ where the $h_i$ and $h_{i,j}$ respect the inclusion maps $f_i$ and $f_{i,j}$ then there exists a unique homomorphism $h:\overline{G}\to H$ which respects the inclusions as well.

Note that because the $G_i$ act transitively, a different choice of $C$ would give stabilizers conjugate to these $B_i$ so that this generalized graph product is unique up to an inner automorphism of $\prod_I G_i$.

So far we have not assumed the $G_i$ act effectively on the $E_i$. Let $N_i$ be the kernel of the action of $G_i$ on $E_i$ and $N=\prod_I N_i$.  Then $N \unlhd  B$.  Further let $\t{G}$ be the group of all lifts of the action of $\prod_I G_i$ on $E$ to $\tE$.  Then $\t{G}$ neccessarily acts effectively on $\tE$.

\begin{lemma}\label{lem:ggplift}
If $\overline{G}$ and $\t{G}$ are as above, then $\t{G} = \overline{G} / N$.
\end{lemma}
\begin{proof}
Fix $\t{C_0}$ a chamber in the lift of $C_0$ and let $\t{B}$ be the lift of $B$ which stabilizes $\t{C_0}$.  Let $\t{H_i}$ be the lift of $H_i$ stabilizing the $i$-panel containing $\t{C_0}$ and for each edge $\{i,j\}$ in the graph $\graph$ let $\t{H_{i,j}}$ be the lift of $H_{i,j}$ stabilizing the $\{i,j\}$-residue containing $\t{C_0}$.  This induces maps from $B$, $H_i$ and $H_{i,j}$ to $\t{G}$ the group of lifts.  It follows from the construction that these maps respect inclusion.  This gives us a unique map $h:\overline{G}\to\t{G}$ that agrees with the embeddings.

The $\t{H_i}$ generate all of $\t{G}$ so that $h$ must be onto.  The map $h$ induces an action of $\overline{G}$ on $\tE$.  Because $\t{G}$ is defined as the group of lifts, $\t{G}$ must act effectively so that the kernel of $h$ is the kernel of the action.  We have an action of each $H_i$ on $\t{E}$ by lifts of the action of $H_i$ on $E$.  The kernel of the action of $H_i$ on $E$ is $N$ the product of the $N_i$, thus $N < \overline{G}$ is the kernel of $h$.  This gives us that $\t{G}=\overline{G}/N$ as desired.
\end{proof}

It immediately follows that $\overline{G}/N \subset \spaut{(\tE)}$.  Note that when the $G_i$ act effectively, $N$ is trivial so these groups are the same.  In most cases of interest these notions coincide.  

\section{Residual Finiteness}

When the $G_i$ are finite we have that $\t{G}$ is a lattice in $\spaut(\tE)$.  We refer to $\t{G}$ as the \emph{standard lattice} of $\spaut(\tE)$.  Margulis proved that a lattice $\Gamma$ of a group $G$ is arithmetic in some group $G$ if and only if the commensurator of $\Gamma$ in $G$ is dense in $G$ \cite{Margulis91}.  Liu proved that the commensurator of the standard uniform lattice of a tree in the full automorphism group  is dense in that group \cite{Liu94}.  Haglund generalized this to the following statement in \cite{haglund}.

\begin{theorem}
If we let $E_i=G_i$ then the commensurator of $\t{G}$ in $\aut(\tE)$ is dense in that group.
\label{comm}
\end{theorem}

In this paper, Haglund also proved that the standard lattice is residually finite.  If we let $\t{G}^0$ be the group of lifts of the identity, we have that $\t{G}^0$ is finite index in $\t{G}$.  Thus, $\t{G}^0$ is also residually finite.

Let the $E_i$ again be finite rank one buildings, and require only that the $G_i$ act transitively on $E_i$.  We apply Haglund's result to the lifts of the identity, $\t{G}^0$ in $\spaut(\tE)$, and see that the action of $\t{G}^0$ on $\tE$ is residually finite.  Now, if the $G_i$ are all finite, then $\t{G}^0$ is finite index in $\t{G}$.  Thus the action of $\t{G}$ on $\tE$ is residually finite.  Hence we have the following result:

\begin{theorem}\label{thm:resfinite}
Let each $\bldg_i$ be finite, and $G_i=\aut(\bldg_i)$ then the standard lattice in $\tG$ is residually finite.
\end{theorem}

\bibliography{biblio}
\bibliographystyle{plain}

\end{document}